\documentclass[a4paper,11pt]{article}

\usepackage[T2A]{fontenc}
\usepackage[cp1251]{inputenc}
\usepackage[english,ukrainian]{babel}
\usepackage{amssymb,upref}
\usepackage{amsmath,amsthm}
\usepackage{multicol}
\usepackage{geometry}
\usepackage[dvips]{graphicx}
\usepackage{fancyhdr}
 \fancypagestyle{plain}{
\lfoot{\footnotesize \copyright\ Д.О. Іваненко,  2014} }

\makeatletter
\renewcommand\section{\@startsection{section}{1}{\z@}%
                                   {3.5ex \@plus 1ex \@minus .2ex}%
                                   {2.3ex \@plus .2ex}%
                                   {\centering\normalfont\bfseries}}

\def\@biblabel#1{#1.}
\makeatother

\newtheorem{thm}{Theorem}
\newtheorem{lem}{Lemma}
\newtheorem{prop}{Proposition}

\newtheorem{rem}{Remark}
\newtheorem{dfn}{Definition}

\newcommand{\eps}{\varepsilon}

\renewcommand{\P}{\mathsf P}
\newcommand{\E}{\mathsf E}
\newcommand{\Ef}{\mathcal E}

\newcommand{\F}{\mathcal F}

\newcommand{\R}{\mathbb R}
\renewcommand{\Re}{\mathbb R}
\newcommand{\df}{\mathrm d}
\newcommand{\prt}{\partial}
\newcommand{\mc}[1]{\mathcal {#1}}
\newcommand{\dif}{{\mathrm D}}
\newcommand{\dom}{{\rm dom}(\dif)}

\newcommand{\be}{\begin{equation}}
\newcommand{\ee}{\end{equation}}

\let\leq\leqslant

\let\geq\geqslant

\begin{document}
\thispagestyle{plain} УДК 519.21
\setlength{\columnsep}{7mm}%
\setlength{\columnseprule}{.4pt}
\begin{multicols}{2}

\begin{flushleft}
Іваненко Д.О.$^{1}$

\vspace{0.1 cm}

{\bfseries Друга похідна логарифмічної функції вірогідності для
моделі заданої СДР керованим процесом Леві}
\end{flushleft}
\vspace{0.2 cm}

$^1$Київський національний університет імені Тараса Шевченка,
факультет радіофізики, електроніки та комп'ютерних систем, 01601,
Київ, вул.~Володимирська 64, \linebreak
e\nobreakdash-mail:
ida@univ.net.ua

\columnbreak

{\selectlanguage{english}
\begin{flushleft}

D.O. Ivanenko$^1$

 \vspace{0.1 cm}

 {\bfseries Second derivative of the log-likelihood in the
model given by a L\'evy driven SDE’S}
\end{flushleft}
 \vspace{0.4 cm}

$^1$Taras Shevchenko National University of Kyiv, Department of
Radio physics, electronics and computer systems, 01601, Kyiv,
Volodymyrska str., 64,\linebreak e\nobreakdash-mail:
ida@univ.net.ua }

\end{multicols}

{\itshape Методами числення Малявена отримано представлення для
другої похідної по параметру логарифмічної функції вірогідності
побудованої на дискретних спостереженнях процесу заданого лінійним
стохастичним диференціальним рівнянням, керованим процесом Леві.

Ключові слова: ММВ, функція вірогідності, СДР, регулярний
статистичний експеримент, ЛАН.}

{\itshape By means of the Malliavin calculus, integral
representation for the second derivative of the loglikelihood
function are given for a model based on discrete time observations
of the solution to equation $
    \df X_t=a_\theta(X_t)\df t +\df Z_t
$ with a  L\'evy process $Z$.

If we have a logarithm of transition kernel for Markov chain and
can calculate two its derivatives w.r.t. parameter, we can find
the maximum likelihood estimate (MLE) and its asymptotic normal
distribution.  But in our case the support of transition
probability density depend on parameter and we can't, in
principle, to obtain a precise formula for the logarithm of joint
density and its derivatives.

The likelihood function in our model is highly implicit. In this
paper, we develop an approach which makes it possible to control
the properties of the likelihood and log-likelihood functions only
in the terms  of the objects involved in the model: the function
$a_\theta(x)$, its derivatives,  and the L\'evy measure of the
L\'evy process $Z$.

 } \vfill

\vspace{0.4 cm}

Статтю представив д.ф.-м.н. Козаченко Ю.В.

\setlength{\columnsep}{4mm}%
\setlength{\columnseprule}{0pt}
\begin{multicols}{2}

\section*{Introduction}
Let  $Z$ be  a L\'evy process without a diffusion component; that
is,
$$
Z_t=ct+\int_0^t\int_{|u|>1}u\nu(\df s, \df
u)+\int_0^t\int_{|u|\leq 1}u\tilde\nu(\df s, \df u),
$$
where $\nu$ is a Poisson point measure with the intensity measure
$\df s \mu(\df u)$, and  $\tilde \nu (\df s, \df u)=\nu(\df s, \df
u)-\df s \mu(\df u)$ is respective compensated Poisson measure. In
the sequel, we assume the L\'evy measure $\mu$ to satisfy the
following:

 \begin{itemize}  \item[\textbf{H.} (i)] for some $\kappa>0$,
$$
\int_{|u|\geq 1}u^{2+\kappa}\mu(du)<\infty;
$$

\item[(ii)] for some $u_0>0$, the restriction of $\mu$ on $[-u_0,
u_0]$ has a positive density $$\sigma\in
C^2\left(\left[-u_0,0\right)\cup \left(0, u_0\right]\right);$$

\item[(iii)] there exists $C_0$ such that
\begin{multline*}
|\sigma'(u)|\leq C_0|u|^{-1}\sigma(u),\\ |\sigma''(u)|\leq
C_0u^{-2}\sigma(u),\ |u|\in (0, u_0];
\end{multline*}

\item[(iv)] $ \left(\log {1\over
\eps}\right)^{-1}\mu\Big(\{u:|u|\geq \eps\}\Big)\to \infty,\quad
\eps\to 0.
$
\end{itemize}
Consider stochastic equation of the form
\begin{equation}\label{eq1}
    \df X_t^\theta=a_\theta(X_t^\theta)\df t +\df Z_t,
\end{equation}
where $a:\Theta\times\R\to \R$ is a measurable function,
$\Theta\subset\R$ is a parametric set.

In \cite{MLE} it was proved that under conditions of smoothness
and growth of $a_\theta$ the Markov process $X$ given by
\eqref{eq1} has a transition probability density $p_t^\theta$
w.r.t. the Lebesgue measure. Besides, according to \cite{MLE} this
density has a derivative $\prt_\theta p_t^\theta(x,y)$. The
extension of the asymptotic methods of mathematical statistics is
used as a key tool the second derivative of the log-likelihood
ratio w.r.t. parameter. The purpose of this paper is to give a
Malliavin-type integral representation of this derivative.

\section{Main results}
We denote by $\P_x^\theta$ the distribution of this process in
$\mathbb{D}([0, \infty))$ with $X_0 =x$,  and  by $\E_x^\theta$
the expectation w.r.t. this distribution. Respective
finite-dimensional distribution for given time moments
$t_1<\dots<t_n$ is denoted by $\P_{x,\{t_k\}_{k=1}^n}^\theta$. On
the other hand, solution $X$ to Eq. (\ref{eq1}) is a random
function defined  on the same probability space $(\Omega, \F, \P)$
with the process $Z$, which depends additionally on the parameter
$\theta$ and the initial value $x=X(0)$. We do not indicate this
dependence in the notation, i.e. write $X_t$ instead of e.g.
$X^\theta_{x,t}$, but it will be important in the sequel that,
under certain conditions, $X_t$ is $L_2$-differentiable w.r.t.
$\theta$ and is $L_2$-continuous w.r.t $(t,x,\theta)$.

In the sequel we will show that, under appropriate conditions,
Markov process $X$ admits a transition probability density
$p^\theta_t(x,y)$ w.r.t. Lebesgue measure, which is continuous
w.r.t. $(t,x,y)\in (0,\infty)\times \Re\times \Re$. Then  (see
\cite{bridges}), for every $t>0, x,y\in \Re$ such that
\be\label{pnon} p^\theta_t(x,y)>0,\ee  there exists a weak limit
in $\mathbb{D}([0, t])$
$$
\P^{t, \theta}_{x,y}=\lim_{\eps\to
0}\P^\theta_x\Big(\cdot\Big||X_t-y|\leq \eps\Big),
$$
which can be interpreted naturally as a \emph{bridge} of the
process $X$ started at $x$ and conditioned to arrive to $y$ at
time $t$. We denote by $\E^{t, \theta}_{x,y}$ the expectation
w.r.t. $\P^{t, \theta}_{x,y}$.

In what follows, $C$ denotes a constant which is not specified
explicitly and may vary from place to place. By
$C^{k,m}(\Re\times\Theta), k,m\geq 0$ we denote the class of
functions $f:\Re\times\Theta\to \Re$ which has continuous
derivatives
$$
\frac{\prt^i}{\prt x^i}\frac{\prt^j}{\prt\ \theta^j}f, \quad i\leq
k, \quad j\leq m.
$$

In \cite{MLE} it was proved that under the conditions of following
Theorem $\prt_\theta p_t^\theta(x,y)$ has a Malliavin-type
integral representation \be\label{d_theta_rep} \prt_\theta
p_t^\theta(x,y)= g_t^\theta(x,y) p_t^\theta(x,y) \ee with
\be\label{g} g_t^\theta(x,y)=\begin{cases}\prt_\theta\log
p_t^\theta(x,y)=\E^{t,\theta}_{x,y}\Xi_t^1, &
p_t^\theta(x,y)>0,\\
0,&\hbox{otherwise}.\end{cases} \ee The goal of this section is to
obtain the same representation for second derivative, i.e.

\be\label{d_theta_rep1} \prt^2_{\theta\theta} p_t^\theta(x,y)=
G_t^\theta(x,y) p_t^\theta(x,y) \ee with \begin{multline}\label{G}
G_t^\theta(x,y)=\\ \begin{cases}\prt^2_{\theta\theta} \log
p_t^\theta(x,y)+g_t^\theta(x,y)^2= \\ =\E^{t,\theta}_{x,y}\Xi_t^2,
&
p_t^\theta(x,y)>0,\\
0,&\hbox{otherwise}.\end{cases} \end{multline} The functionals
$\Xi^{1}_{t}$ and $\Xi^{2}_{t}$, involved in expressions for $g$
and $G$, will be introduced explicitly in the proof below; see
formulas \eqref{Xi_1} and \eqref{Xi_2}.

\begin{thm}\label{mainthm2} Let $a\in C^{3,2} (\Re\times\Theta)$ have bounded derivatives
$\prt_xa$, $\prt^2_{xx}a$, $\prt^2_{x\theta}a$, $\prt^3_{xxx}a$,
$\prt^3_{x\theta\theta}a$, $\prt^3_{xx\theta}a$,
$\prt^4_{xxx\theta}a$ and for all $\ \theta\in \Theta, \ x\in \R$
\be\label{lin_gr1} |a_\theta(x)|+|\partial_{\theta}
a_\theta(x)|+|\partial^2_{\theta\theta} a_\theta(x)|\leq
C(1+|x|). \ee Then the transition probability density has a second
derivative $\prt^2_{\theta\theta} p_t^\theta(x,y),$ which is
continuous w.r.t. $(t,x,y, \theta)\in (0, \infty)\times \Re\times
\Re\times \Theta$, and \eqref{d_theta_rep1} holds true.
\end{thm}

\begin{rem} By statement of Theorem, the logarithm of  the transition
probability density has a second continuous derivative w.r.t.
$\theta$ on the open subset of $(0,\infty)\times \Re\times
\Re\times \Theta$ defined by inequality $ p^\theta_t(x,y)>0$ and,
on this subset, admits the integral representation \be\label{log}
\prt^2_{\theta\theta} \log
p_t^\theta(x,y)=\E^{t,\theta}_{x,y}\Xi_t^2-\left(E^{t,\theta}_{x,y}\Xi_t^1\right)^2.
\ee
\end{rem}

\begin{rem}\label{dg_momres} For every $\gamma<1+\kappa/2$ there
exists constant $C$ which depends on $t$ and $\gamma$ only, such
that \be\label{dg_mom} \E_x^\theta \Big|\prt_\theta g_t^\theta(x,
X_t^\theta)\Big|^\gamma\leq C(1+|x|)^\gamma. \ee
\end{rem}

\section{Proof of Theorem \ref{mainthm2}}

We need to repeat some notations and statements defined in Section
3 \cite{MLE}. Fix $u_1\in (0, u_0)$, where $u_0$ comes from
\textbf{H} (ii), and introduce a $C^2$-function $\varrho:\Re\to
\Re^+$ with bounded derivative, such that $$
\varrho(u)=\begin{cases} u^{2},&|u|\leq u_1;\\
0,&|u|\geq u_0\end{cases}.
$$  Denote by $Q_c(x),\ c\in \Re$  the value at the time
moment $s=c$ of the solution to  Cauchy problem
$$
q'(s)=\varrho(q(s)), \quad q(0)=x.
$$
Then $\{Q_c, c\in \Re\}$ is a group of transformations of $\Re$,
and $\prt_c Q_c(x)|_{c=0}=\varrho(x)$.

\begin{dfn}\label{def1} A functional  $F\in L_2(\Omega, \F, \P)$
is called \emph{stochastically differentiable}, if there exists an
$L_2(\Omega, \F, \P)$-limit \begin{equation}\label{dif} \hat\dif
F=\lim_{c\to 0}{1\over c}\Big(\mc{Q}_cF-F\Big).
\end{equation}
The closure $\dif$ of the operator $\hat \dif$ defined by
(\ref{dif}) is called the \emph{stochastic derivative}.  The
adjoint operator $\delta=\dif^*$ is called the \emph{divergence
operator} or the \emph{extended stochastic integral}.
\end{dfn}

\begin{rem}\label{remark01} $\dom$ is dense in
$L_2(\Omega, \F, \P)$, hence $\delta$ is well defined. In
addition, $\mathrm{dom}(\delta)$ is dense in $L_2(\Omega, \F,
\P)$, hence $\hat \dif$ is closable.
 The operator $\delta$ itself is closed as an adjoint one; e.g. Theorem VIII.1
 in \cite{gorod}.
\end{rem}

Denote $\chi(u)=-{(\sigma(u)\varrho(u))'\over \sigma(u)}$,
$u\not=0$.

\begin{prop}\label{lem01} 1. Let $\varphi\in C^1(\Re^d, \Re)$ have bounded
derivatives and  $F_k\in\dom$, $k=\overline{1,d}$. Then
 $\varphi(F_1,\dots,F_d)\in\dom$ and \be\label{chain}\dif\left[
\varphi(F_1,\dots,F_d)\right]=\sum\limits_{k=1}^d[\partial_{x_k}\varphi](F_1,\dots,F_d)\dif
F_k.\ee

2. The constant function $1$ belongs to $\mathrm{dom}(\delta)$ and
\be\label{delta_1} \delta(1)=\int_0^T\int_{\Re}\chi(u)\tilde
\nu(\df s,\df u). \ee

3. Let  $G\in\dom$ and
\begin{equation}\label{2}
    \E\left(\delta(1)G\right)^2<\infty.
\end{equation}

Then $G\in\mathrm{dom}(\delta)$ and $\delta(G)=\delta(1)G-\dif G.
$
\end{prop}

The proofs of this Proposition and Remark \ref{remark01} can be
found in \cite{MLE}.

\begin{lem}\label{D^kX} Under the conditions of Theorem \ref{mainthm2}
$X_t^\theta$ is thrice stochastically differentiable and
\begin{multline}\label{DX}
\dif^jX_t^\theta=\sum_{i=0}^{j-1}\frac{(i+1)^{j-i+1}}{i!}
\int_0^t\dif^{j-i-1}\left(\Ef_t\Ef_s^{-1}\right)\\
\int_{\R}\varrho(u)\left(\varrho(u)^{i}\right)^{(i)}\nu(\df s, \df
u), \ j=\overline{1,3};\end{multline} where
$\Ef_t:=\exp\left\{\int_0^t\partial_xa_\theta( X_\tau^\theta)\df
\tau\right\}$,
\begin{multline}\label{DEf}\dif^n\Ef_t=\sum_{k=0}^{n-1}\sum_{j=0}^{n-k-1}C_{n-1}^kC_{n-k-1}^j
\dif^k
\Ef_t\times\\\int_0^t\dif^j\left(\prt_{xx}^2a_\theta(X_\tau^\theta)\right)
\dif^{n-k-j}X_\tau^\theta\df\tau, \ n=1,2.\end{multline}
\end{lem}

\begin{rem} The expressions for $\dif^n\left(\prt^2_{xx}
a_\theta(X_t^\theta)\right)$ and $\dif^n\left(\Ef_t
\Ef_s^{-1}\right)$ can be found by the first statement of
Proposition \ref{lem01} (and formula \eqref{DEf} respectively).
\end{rem}

\begin{rem} Under additional conditions about smoothness and growth of
$a_\theta$ the formulas \eqref{DX} and \eqref{DEf} are equitable
if $j$ is more than 3 and $n$ is more than 2.
\end{rem}

The case $j=1,2$ and $n=1$ was considered in \cite{MLE}. The proof
of \eqref{DX} as $j\geq3$ provides by induction using the argument
of proof of relation (27) \cite{MLE}, and based on Theorem
II.2.8.5 \cite{scorohod2}. The same arguments that in Section 3.2
\cite{MLE} give (see details in proof of relations (27), (31) and
(32) \cite{MLE}):
\begin{multline}\label{Dpr1X}
\dif^2\partial_{\theta\theta}X_t^\theta=\int_0^t\dif^2
\left(\Ef_t\Ef_s^{-1}\right)\partial_{\theta}a_\theta(X_s^\theta)\df
s+\\2\int_0^t\dif \left(\Ef_t\Ef_s^{-1}\right)\partial^2_{x\theta}
a_\theta(X_s^\theta)\dif X_s^\theta\df s+\\
\Ef_t\int_0^t\Ef_s^{-1}\left(\partial^3_{xx\theta}
a_\theta(X_s^\theta)(\dif
X_s^\theta)^2+\right.\\\left.\partial^2_{x\theta}
a_\theta(X_s^\theta)\dif^2 X_s^\theta\right)\df s,
\end{multline}
\begin{multline}\label{d2X}
\partial^2_{\theta\theta}X_t^\theta=\Ef_t\int_0^t\Ef_s^{-1}
\left([\partial^2_{\theta\theta}a_\theta](X_s^\theta)+\right.\\\left.
2[\partial^2_{x\theta}a_\theta](X_s^\theta)\partial_\theta
X_s^\theta+[\partial^2_{xx}a_\theta](X_s^\theta)(\partial_\theta
X_s^\theta)^2\right)\df s, \end{multline}
\begin{multline}\label{DprX}
\dif\partial^2_{\theta\theta}X_t^\theta=2\Ef_t\int_0^t\Ef_s^{-1}
\left(\partial^2_{xx}a_\theta(X_s^\theta)\partial_\theta
X_s^\theta+\right.\\\left.[\partial^2_{x\theta}a_\theta](X_s^\theta)\right)
\dif\prt_\theta X_s^\theta\df s+
\\
\Ef_t\int_0^t\Ef_s^{-1}\left(2\partial^3_{xx\theta}[a_\theta](X_s^\theta)\partial_\theta
X_s^\theta+\right.\\\left.\partial^3_{xxx}a_\theta(X_s^\theta)(\partial_\theta
X_s^\theta)^2+[\partial^3_{x\theta\theta}a_\theta](X_s^\theta)\right)
\dif X_s^\theta\df s.
\end{multline}

Similarly to proof the moment bounds for $\prt_\theta X_t^\theta$,
 $\dif(\prt_\theta X_t)$,
 $\dif X_t^\theta$, $\dif^2 X_t^\theta$, proved in Section
3.3 \cite{MLE}, we get the same one for $\prt^2_{\theta\theta}
X_t^\theta$, $\dif(\prt^2_{\theta\theta} X_t^\theta)$,
$\dif^2(\prt_\theta X_t^\theta)$ and $\dif^3 X_t^\theta$. Note
that the assumption on the derivatives $\prt_xa, \prt_{xx}^2a,
\prt^3_{xxx}a$ is used in Section 3.2 \cite{MLE} to get the
existence of the derivatives $\dif X_t, \dif^2 X_t, \dif^3 X_t$.
The addition assumption on $\prt^2_{x\theta}a_\theta$ similarly
gives the existence of derivative $\dif\left(\prt_\theta
X_t^\theta\right)$.

\begin{proof}[Proof of Theorem \ref{mainthm2}] In the theorem 1
\cite{MLE} it was proved that the transition
probability density has a derivative  $\prt_\theta
p_t^\theta(x,y),$ which is continuous  w.r.t. $(t,x,y, \theta)\in
(0, \infty)\times \Re\times \Re\times \Theta$, and functional
$\Xi_t^1$, from its representation given by the formula

\be\label{Xi_1} \Xi_t^1={(\prt_\theta X_t^\theta) \delta(1)\over
\dif X_t^\theta}+{(\prt_\theta X_t^\theta) \dif^2 X_t^\theta \over
(\dif X_t^\theta)^2}-{\dif(\prt_\theta X_t^\theta)\over \dif
X_t^\theta}. \ee Note that  $X_t$ is twice $L_2$-differentiable
w.r.t. parameter $\theta$, see \eqref{d2X} for its derivative. In
addition, $\dif X_t^\theta$, $\dif^2 X_t^\theta$, and  $\dif
\prt_\theta X_t^\theta$, are $L_2$-differentiable w.r.t. $\theta$,
and all these derivatives satisfy moment bounds similar to (35)
\cite{MLE} (moment bounds for $\dif X_t^\theta$). Now it is easy
to prove that $\Xi_t^1$ is $L_2$-differentiable w.r.t. $\theta$
(the explicit formula of the derivative is omitted). One can just
replace $\dif X_t$ in the denominator in the formula \eqref{Xi_1}
by $\dif X_t+\eps$, prove that this new functional is
$L_2$-differentiable w.r.t. $\theta$ using the chain rule, and
then show using (36) \cite{MLE} (negative order moment bounds for
$\dif X_t^\theta$) that both this functional and its derivative
w.r.t. $\theta$ converge (locally uniformly) in $L_2$ as $\eps\to
0$, respectively,  to $\Xi_t^1$ and to the functional
$\prt_\theta\Xi_t^1$ which comes from the formal differentiation
of \eqref{Xi_1}. This argument also shows that $\Xi_t^1$ and
$\prt_\theta\Xi_t^1$ depend continuously (in $L_2$) on $x, t,
\theta$. Therefore, we can take a derivative at the right hand
side in \eqref{d_theta_rep}, which gives
$$
\prt^2_{\theta\theta}p_t^\theta(x,y)=
p_t^\theta(x,y)\E^{t,\theta}_{x,y}\prt_\theta\Xi_t^1+p_t^\theta(x,y)g_t^\theta(x,y)^2.
$$
This function is continuous w.r.t. $(t,x, y,\theta)$ because
$p_t^\theta$, $g_t^\theta$, and $\prt_\theta\Xi_t^1$ depend
continuously (in $L_2$) on $x, t, \theta$, and relation
\be\label{null} \P_x^\theta(X_t=y)=0, \quad x, y\in \Re, \quad
t>0, \quad \theta\in \Theta \ee holds true (by representation
\eqref{d_theta_rep}).

To prove \eqref{d_theta_rep1}, we use moment bounds for
$\prt_\theta X_t^\theta$, $\prt^2_{\theta\theta} X_t^\theta$,
$\dif(\prt_\theta X_t)$, $\dif(\prt^2_{\theta\theta} X_t^\theta)$,
$\dif^2(\prt_\theta X_t^\theta)$, $\dif X_t^\theta$, $\dif^2
X_t^\theta$ and $\dif^3 X_t^\theta$ to get, similarly to the proof
of (37) \cite{MLE} (integral representation for $p_t^\theta$),
that
$$\frac{(\prt_\theta X_t^\theta)^2}{\dif X_t^\theta},\quad
\frac{1}{\dif X_t^\theta}\left(\delta\left(\frac{(\prt_\theta
X_t^\theta)^2}{\dif X_t^\theta}\right)+\partial^2_{\theta\theta}
X_t^\theta\right)$$ belong to $\mathrm{dom}(\delta)$ and

\begin{multline}\label{Xi_2} \Xi_t^2:=\delta\left(\frac{1}{\dif
X_t^\theta}\left(\delta\left(\frac{(\prt_\theta
X_t^\theta)^2}{\dif X_t^\theta}\right)+\partial^2_{\theta\theta}
X_t^\theta\right)\right)=\\-{1 \over \dif
X_t^\theta}\dif\delta\left({(\prt_\theta X_t^\theta)^2 \over \dif
X_t^\theta}\right)+{\dif\prt^2_{\theta\theta}X_t^\theta\over\dif
X_t^\theta}+
\\ \left({\delta(1)\over \dif X_t^\theta}+{\dif^2
X_t^\theta \over (\dif
X_t^\theta)^2}\right)\left(\delta\left(\frac{(\prt_\theta
X_t^\theta)^2}{\dif X_t^\theta}\right)+\partial^2_{\theta\theta}
X_t^\theta\right),
\end{multline}
with
\begin{multline*}
\delta\left(\frac{(\prt_\theta X_t^\theta)^2}{\dif
X_t^\theta}\right)=\\\frac{(\prt_\theta X_t^\theta)^2 \delta(1)}{
\dif X_t^\theta}+\frac{(\prt_\theta X_t^\theta)^2 \dif^2
X_t^\theta}{(\dif X_t^\theta)^2} -\frac{2(\prt_\theta
X_t^\theta)\dif(\prt_\theta
X_t^\theta)}{\dif X_t^\theta},
\end{multline*}
\begin{multline*}
\dif\delta\left(\frac{(\prt_\theta X_t^\theta)^2}{\dif
X_t^\theta}\right)= \\{2\prt_\theta X_t^\theta\over\dif
X_t^\theta}\left(\delta(1)\dif(\prt_\theta
X_t^\theta)-\dif^2(\prt_\theta X_t^\theta)\right)+{(\prt_\theta
X_t^\theta)^2\dif\delta(1)\over\dif
X_t^\theta}\\-{2(\dif(\prt_\theta X_t^\theta))^2\over\dif
X_t^\theta}+ \left({\prt_\theta X_t^\theta\over\dif
X_t^\theta}\right)^2 \left(\dif^3 X_t^\theta-\delta(1)\dif^2
X_t^\theta\right)\\+{4\prt_\theta X_t^\theta\dif(\prt_\theta
X_t^\theta)\dif^2 X_t^\theta\over(\dif
X_t^\theta)^2}-{2(\prt_\theta X_t^\theta\dif^2
X_t^\theta)^2\over(\dif X_t^\theta)^3}.
\end{multline*}
The expressions for $\prt_\theta X_t^\theta$, $\dif\prt_\theta
X_t^\theta$ and $\dif\delta(1)$ can be found in \cite{MLE}, the
other one given by the formulas \eqref{DX} -- \eqref{DprX}.
Therefore, for any test function $f\in C^2(\Re)$ with bounded
derivatives we have
\begin{multline}\label{rightside}
    \partial^2_{\theta\theta}\E_x^\theta f(X_t^\theta)=\\
    \E_x^\theta \left(f''(X_t^\theta)(\prt_\theta X_t^\theta)^2+f'(X_t^\theta)
    \partial^2_{\theta\theta} X_t^\theta\right)=\\
    \E_x^\theta \left(\dif f'(X_t^\theta){(\prt_\theta X_t^\theta)^2\over \dif X_t^\theta}
    +f'(X_t^\theta)\partial^2_{\theta\theta} X_t^\theta\right)=\\
    \E_x^\theta \left(f'(X_t^\theta)\left(\delta\left({(\prt_\theta X_t^\theta)^2\over \dif
    X_t^\theta}\right)
    +\partial^2_{\theta\theta} X_t^\theta\right)\right)=\\
    \E_x^\theta \left(\frac{\dif f(X_t^\theta)}{\dif X_t^\theta}
    \left(\delta\left(\frac{(\prt_\theta X_t^\theta)^2}{\dif
    X_t^\theta}\right)+\partial^2_{\theta\theta} X_t^\theta
    \right)\right)=\\
    \E_x^\theta f(X_t^\theta)\Xi^2_t=\E_x^\theta f(X_t^\theta)G_t^\theta(x,X_t^\theta);
\end{multline}
see \eqref{G} for the definition of $G_t^\theta(x,y)$. Because the
test function $f$ is arbitrary, the integral identity
(\ref{rightside}) proves (\ref{d_theta_rep1}).
\end{proof}

\begin{rem} From \eqref{rightside} with $f\equiv 1$ it follows
that for every $x\in \Re, \theta\in \Theta, t>0$
$$
\E_x^\theta G_t^\theta(x, X_t^\theta)=0.
$$
\end{rem}

\begin{proof}[Proof of Remark \ref{dg_momres}]
By the moment bounds and formula \eqref{Xi_2}, we have
\be\label{Xi_mom_1} \E_x^\theta|\Xi_t^2|^{p}\leq C(1+|x|^p) \ee
for every $p\in [1, 2+\kappa)$, with the constants $C$ depending
on $t, p$ only.

Combining relations \eqref{d_theta_rep} -- \eqref{G} we get
$$
\prt_\theta g_t^\theta(x,
X_t)=\E_x^\theta\Big[\Xi_t^2\Big|X_t\Big]-g_t^\theta(x,
X_t^\theta)^2,
$$
Moreover, inequality \eqref{dg_mom} follows directly from
\eqref{Xi_mom_1}, (45) \cite{MLE} (moment bounds for $g_t^\theta$)
and Jensen's inequality.
\end{proof}

\end{multicols}

\begin{multicols}{2}

\renewcommand{\refname}{Список використаних джерел}

\renewcommand{\refname}{References}

\begin{flushright}
Received:  23.12.2013
\end{flushright}
\end{multicols}
\end{document}